\documentclass[12pt]{amsart}
\usepackage{amsfonts}
\usepackage{amssymb}
\usepackage{bm}

 \newtheorem{theorem}{Theorem}[section]
 
 \newtheorem{conjecture}[theorem]{Conjecture}
 \newtheorem{lemma}[theorem]{Lemma}
 \newtheorem{proposition}[theorem]{Proposition}
 \theoremstyle{definition}

 \newtheorem{remark}[theorem]{Remark}

\numberwithin{equation}{section}

\begin{document}
\title[On a conjecture of Zhuang  and Gao]
{On a conjecture of Zhuang  and Gao}

\author[Y.K. Qu]{Yongke Qu}
\address{Department of Mathematics\\ Luoyang Normal University\\
LuoYang 471934, P.R.
CHINA}
\email{yongke1239@163.com}

\author[Y.L. Li]{Yuanlin Li*}
\address{Department of Mathematics and Statistics\\  Brock University\\ St. Catharines, ON L2S 3A1, Canada}
\email{yli@brock.ca}

\thanks{*Corresponding author: Yuanlin Li, E-mail: yli@brocku.ca}

\begin{abstract}
Let $G$ be a multiplicatively written finite group. We denote by $\mathsf E(G)$ the smallest integer $t$ such that every sequence of $t$ elements in $G$ contains a product-one subsequence of length $|G|$. In 1961, Erd\H{o}s, Ginzburg and Ziv proved that $\mathsf E(G)\leq 2|G|-1$ for every finite ablian group $G$ and this result is known as the Erd\H{o}s-Ginzburg-Ziv Theorem. In 2005, Zhuang and Gao conjectured that $\mathsf E(G)=\mathsf d(G)+|G|$, where $\mathsf d(G)$ is the small Davenport constant. In this paper, we confirm the conjecture for the case when $G=\langle x, y| x^p=y^m=1, x^{-1}yx=y^r\rangle$, where $p$ is the smallest prime divisor of $|G|$ and $\mbox{gcd}(p(r-1), m)=1$.
\end{abstract}

\date{}

\maketitle


\section{Introduction and main results}
Let $G$ be a finite group written multiplicatively. Let $S=g_1\bm\cdot \ldots\bm\cdot g_{\ell}$ be a sequence over $G$ with length $\ell$. We use $$\pi(S)=\{g_{\tau (1)}\ldots g_{\tau (\ell)}: \tau \mbox { a permutation of } [1,\ell]\}\subseteq G$$ to denote the set of products of $S$. We say that $S$ is a {\sl product-one sequence} if $1\in\pi(S)$. Let $\mathsf d(G)$ be the {\sl small Davenport constant} of $G$ (i.e., the maximal integer $\ell$ such that there is a sequence of length $\ell$ over $G$ which has no nontrivial product-one subsequence). We denote by $\mathsf E(G)$ the smallest integer $t$ such that every sequence of length $t$ over $G$ contains a product-one subsequence of length $|G|$. The problem of finding the precise value of the Davenport constant and what is now known as the Erd\H{o}s-Ginzburg-Ziv Theorem have become the starting points of zero-sum theory.

\begin{theorem}(Theorem of Erd\H{o}s-Ginzburg-Ziv) \cite{EGZ1961}\label{EGZ}
Let $G$ be a finite abelian group. Then $\mathsf E(G)\leq 2|G|-1$.
\end{theorem}
Since that time (dating back to the early 1960s), zero-sum theory has developed into a flourishing branch of additive and combinatorial number theory. At the present moment, two conjectures on $\mathsf E(G)$ can be found in the literature. We briefly discuss some history and the motivation for studying these conjectures. For more detailed information, the interested reader may refer to the surveys \cite{Caro, GG2006, GeR} or the monographs \cite{GeK,Gr}. Although the main focus of zero-sum theory has initially been on abelian groups, the research has never been restricted to the abelian setting alone. In 1976, Olson \cite{O} showed that $\mathsf E(G)\leq 2|G|-1$ holds for all finite groups. In 1984, Yuster and Peterson \cite{YP1984} showed that $\mathsf E(G)\leq 2|G|-2$ when $G$ is a non-cyclic solvable group. Later Yuster \cite{YP1988} improved the result to $\mathsf E(G)\leq 2|G|-r$ provided that $|G|\geq 600((r-1)!)^2$. In 1996, Gao \cite{Gao1996} further improved the upper bound to $\mathsf E(G)\leq \frac{11|G|}{6}-1$. More recently in 2010, Gao and Li \cite{GL2010} proved that $\mathsf E(G)\leq\frac{7|G|}{4}-1$ and they proposed the following conjecture.

\begin{conjecture}\label{EGZconjecture}
Let $G$ be any finite non-cyclic group. Then $\mathsf E(G)\leq \frac{3|G|}{2}$.
\end{conjecture}

We remark that the above mentioned upper bound on $\mathsf E(G)$ is the best possible since $\mathsf E(G)=\frac{3|G|}{2}$ when $G= D_{2n}$ is a dihedral group \cite[Theorem 8] {Bass2007}.

In 1996, Gao \cite{G1996} discovered the fundamental relationship between $\mathsf E(G)$ and $\mathsf d(G)$ for abelian groups, i.e., $\mathsf E(G)=\mathsf d(G)+|G|$. In 2005, Zhuang and Gao \cite{ZG2005} investigated this relation $\mathsf E(G)=\mathsf d(G)+|G|$ for some non-abelian groups and proposed the following conjecture.

\begin{conjecture}\label{mainconj}
Let $G$ be a finite group. Then $\mathsf E(G)=\mathsf d(G)+|G|$.
\end{conjecture}

They confirmed Conjecture~\ref{mainconj} for dihedral groups of order $2p$ where $p\geq 4001$ is a prime. Gao and Lu \cite{GL2008} then improved this result to dihedral groups of order $2n$ for all $n\geq 23$. Bass \cite{Bass2007} extended the method used by Gao and Lu, and proved the conjecture for all dihedral groups, dicyclic groups, and groups $C_p\ltimes C_q$, where $p$, $q$ are primes. Recently, in 2015 Han \cite{Han2015} verified the conjecture for the case when $G$ is a non-cyclic nilpotent group or when $G\cong C_p\ltimes C_{pn}$, where $p$ is a prime. Most recently, in 2019 Han and Zhang \cite{HZ2019} further extended Han's early result and verified the conjecture for $G\cong C_m\ltimes C_{mn}$. It is natural to extend this line of research to the class of metacyclic groups $C_n\ltimes C_m$ where $\mbox{gcd}(n,m)=1$. In this paper, we confirm Conjecture~\ref{mainconj} for the following class of metacyclic groups and our main result is as follows

%We found the investigation of a special class of meta-cyclic group $G=\langle x, y| x^p=y^m=1, x^{-1}yx=y^r\rangle\cong C_{p}\ltimes C_m$ where $p$ is the %smallest prime divisor of $|G|$ and $\mbox{gcd}(p(r-1), m)=1$ is of most interest as it is related to resolving Conjecture~\ref{EGZconjecture}. We remark that %most recently, Gao, Li and Qu \cite{GLQ} made a substantial progress on Conjecture~\ref{EGZconjecture} by proving that $\mathsf E(G)\leq \frac{3|G|}{2}$ holds %for all non-cyclic groups of odd order. The method developed in this paper has been played an important role in proving Conjecture~\ref{EGZconjecture} in %\cite{GLQ}. The main result of this paper is as follows.

\begin{theorem}\label{maintheorem} Let $G=\langle x, y| x^p=y^m=1, x^{-1}yx=y^r\rangle$, where $p$ is the smallest prime divisor of $|G|$ and $\mbox{gcd}(p(r-1), m)=1$. Then $\mathsf d(G)=m+p-2$ and $\mathsf E(G)=\mathsf d(G)+|G|$.
\end{theorem}

We remark that the investigation of the above mentioned class of metacyclic groups is of special interest as it is closely related to solving Conjecture~\ref{EGZconjecture}. Very recently, Gao, Li, and Qu \cite{GLQ} made substantial progress on this conjecture and  proved that $\mathsf E(G)\leq \frac{3|G|}{2}$ holds for all non-cyclic groups of odd order by using the minimal counterexample method. Our main result and some technical lemmas in Section 3 have been used in \cite{GLQ} as the necessary ingredients to construct a minimal counterexample leading to a contradiction. In addition, the methods developed in this paper (such as applying a generalization of Kneser's Additive Theorem to estimate a partial product of sets $\Pi^{m}(\mathbf{A})$) have been used repeatedly in \cite{GLQ} to help simplify some complicated computations and to achieve the best possible upper bound on $\mathsf E(G)$.

\section{Notation and Preliminaries}
We follow the notation and conventions detailed in \cite{GG2013}.

For real numbers $a,b\in \mathbb{R}$, we set $[a,b]=\{x\in \mathbb{Z}: a\leq x\leq b\}$. For integers $m,n\in \mathbb{Z}$, we denote by $\mbox{gcd}(m,n)$ the greatest common divisor of $m$ and $n$.

Let $G$ be a finite multiplicative group. If $A \subseteq G$ is a nonempty subset, then denote by  $\langle A \rangle $  the subgroup  of $G$ generated by $A$. If $A$ and $B$ are subsets of $G$, we define the product-set as $AB=\{ab:a\in A,b\in B\}$. Recall that by a {\sl sequence} over a group $G$, we mean a finite, unordered sequence where repetition of elements is allowed. We view sequences over $G$ as elements of the free abelian monoid $\mathcal{F}(G)$, denote multiplication in $\mathcal{F}(G)$ by the bold symbol $\bm\cdot$ rather than by juxtaposition, and use brackets for all exponentiation in $\mathcal{F}(G)$.

A sequence $S \in \mathcal F(G)$ can be written in the form $S= g_1  \bm \cdot \ldots \bm\cdot g_{\ell},$ where $|S|= \ell$ is the {\it length} of $S$. For $g \in G$, let $\mathsf v_g(S) = |\{ i\in [1, \ell] : g_i =g \}|\,  $ denote the {\it multiplicity} of $g$ in $S$. A sequence $T \in \mathcal F(G)$ is called a {\it subsequence } of $S$ and is denoted by $T \mid S$ if  $\mathsf v_g(T) \le \mathsf v_g(S)$ for all $g\in G$. Denote by $S \bm\cdot T^ {[-1]}$  the subsequence of $S$ obtained by removing the terms of $T$ from $S$.

If $S_1, S_2 \in \mathcal F(G)$, then $S_1 \bm\cdot S_2 \in \mathcal F(G)$ denotes the sequence satisfying $\mathsf v_g(S_1 \bm\cdot S_2) = \mathsf v_g(S_1 ) + \mathsf v_g( S_2)$ for all $g \in G$. For convenience, we  write
\begin{center}
 $g^{[k]} = \underbrace{g \bm\cdot \ldots \bm\cdot g}_{k} \in \mathcal F(G)\quad$
\end{center}
for $g \in G$ and $k \in \mathbb{N}_0$.

Suppose $S= g_1 \bm \cdot \ldots \bm\cdot g_{\ell} \in \mathcal F(G)$. Let $$\pi (S) = \{g_{\tau(1)}\ldots g_{\tau(\ell)}: \tau \mbox{ a permutation of } [1, \ell] \} \subseteq G$$ denote the {\it set of products} of $S$. Let $$\Pi_n(S) = \cup _{T\mid S,\ |T| = n}\pi(T)$$ denote the {\it set of all $n$-products} of $S$.
Let $$\Pi(S) = \cup_{1 \le n \le \ell}\Pi_n(S)$$ denote the {\it set of all subsequence products} of $S$. The sequence $S$ is called
\begin{itemize}
\item[$\bullet$]  {\it product-one} if $1 \in \pi(S)$;
\item[$\bullet$] {\it product-one free} if $1\not\in \Pi(S)$;
\item[$\bullet$] {\it minimal product-one } if $1\in\pi(S)$ and $S$ has no proper product-one subsequence.
\end{itemize}

If $\mathbf{A}=(A_1,\ldots, A_{\ell})$ is a sequence of finite subsets of $G$, let $m\leq \ell$, we define $$\Pi^{m}(\mathbf{A})=\{a_{i_1}\ldots a_{i_{m}}:1\leq i_1<\cdots<i_{m}\leq \ell\mbox{ and } a_{i_j}\in A_{i_j} \mbox{ for every } 1\leq j\leq m\}.$$  Let $A$ be a subset of $G$ and $\mbox{stab}(A)=\{g\in G: gA=A\}$ its stabilizer. The following lemma is a generalization of Kneser's Theorem which is crucial for our proof of the main result.

\begin{lemma}\cite[Theorem 1.3]{DGM2009}\cite[Theorem 13.1]{Gr}\label{genKneser}\label{genKneser}
Let $\mathbf{A}=(A_1,\ldots, A_{\ell})$ be a sequence of finite subsets of an abelian group $G$, let $m\leq \ell$, and let $H=\mbox{stab}(\Pi^{m}(\mathbf{A}))$. If $\Pi^{m}(\mathbf{A})$ is nonempty, then $$|\Pi^{m}(\mathbf{A})|\geq |H|\bigg(1-m+\sum_{Q\in G/H}min\big\{\ell,|\{i\in[1,\ell]:A_i\cap Q\neq \emptyset\}|\big\}\bigg).$$
\end{lemma}
We also need the following lemma about product-one free sequences.

\begin{lemma}\cite[Lemma 2.4]{GLP2014}\label{zerofree}
Let $S$ be a product-one free sequence over $G$. Then $|\Pi(S)|\geq |S|$.
\end{lemma}

\section{Proof of Theorem \ref{maintheorem}}
Throughout this section, let $G=\langle x, y| x^p=y^m=1, x^{-1}yx=y^r\rangle\cong C_p\ltimes C_m$, where $p$ is the smallest prime divisor of $|G|$, and $\mbox{gcd}(p(r-1), m)=1$. Let $K=\langle x\rangle$ and $N=\langle y\rangle$. Then $C_{m}\cong N\lhd G$ and $K\cong G/N\cong C_p$. Let $\varphi$ be the canonical homomorphism from $G$ onto $G/N$. Then for each sequence $T$ over $G$, $\varphi(T)$ is a sequence over $G/N$. Note that if $\varphi(T)$ is a product-one sequence over $G/N$, then $\pi(T)\subseteq N$, as $\varphi(\pi(T))=\pi(\varphi(T))=1$. We first prove a few useful lemmas.

\begin{lemma}\label{basic} Let $M$ be any subgroup of $N=\langle y\rangle$, $u$ be an element of $N$ and $0\leq s<s'\leq p-1$. Then
\begin{itemize}
\item[(i)] If $u^{r^s}\in M$, then $\{u, u^{r},\ldots, u^{r^{p-1}}\}\subseteq M$.
\item[(ii)] If both $u^{r^s}$ and $u^{r^{s'}}$ are in the same coset of $M$, then $\{u, u^{r},\ldots, u^{r^{p-1}}\}\subseteq M$.
\item[(iii)] If $u\neq 1$, then $u^{r^s}\neq u^{r^{s'}}$.
\end{itemize}
\end{lemma}
\begin{proof} (i) Since $\mbox{gcd}(r,m)=1$, we have $\mbox{gcd}(r^s, m)=1$. Therefore, there exist $k, e\in \mathbb{Z}$ such that $1=r^sk+me$.  From $u\in N$ we have $u^m=u^{|N|}=1$. If $u^{r^s}\in M$, then $u=u^{r^sk+me}=(u^{r^s})^k\in M$. Thus $\{u, u^{r},\ldots, u^{r^{p-1}}\}\subseteq M$.

(ii) Since both $u^{r^s}$ and $u^{r^{s'}}$ are in the same coset of $M$, $u^{r^s-r^{s'}}\in M$, and thus $u^{(1-r^{s'-s})r^s}\in M$. Since $u^{1-r^{s'-s}}\in N$, by (i) we conclude that $u^{1-r^{s'-s}}\in M$. Next we show that $\mbox{gcd}(1-r^{s'-s},m)=1$. Assume to the contrary that $\mbox{gcd}(1-r^{s'-s},m)\neq 1$. Then there exists a prime divisor $q$ of $m$ such that $1-r^{s'-s}\equiv 0 \pmod q$. Since $\mbox{gcd}(r-1,m)=1$, we have $r\not\equiv 1 \pmod q$. Since $r^p\equiv 1\pmod m$, we have $r^p\equiv 1\pmod q$, which together with $r\not\equiv 1 \pmod q$ and $r^{s'-s}\equiv 1 \pmod q$ gives $s'-s \equiv 0\pmod p$. Thus $s=s'$, yielding a contradiction. Hence, we must have $\mbox{gcd}(1-r^{s'-s},m)=1$. As in (i), we obtain $u\in M$, and thus $\{u, u^{r},\ldots, u^{r^{p-1}}\}\subseteq M$.

(iii) Assume to the contrary that $u^{r^s}=u^{r^{s'}}$. Then both $u^{r^s}$ and $u^{r^{s'}}$ are in the same coset of the trivial subgroup $\{1\}$. By (ii), $u\in \{1\}$, so $u=1$, yielding a contradiction.
\end{proof}

\begin{lemma}\label{conjugationone} Let $T=g_1\bm\cdot \ldots \bm\cdot g_t$ be a sequence over $G$ such that $\varphi (T)$ is a minimal product-one sequence over $G/N$. Then for any $u\in\pi(T)$, we have $\pi(T)\supseteq \{u^{r^{s_{1}}}, u^{r^{s_{2}}}, \ldots, u^{r^{s_{t}}}\}$ for some subset $\{s_1,\ldots, s_t\}\subseteq [0,p-1]$. Moreover, if $u\neq 1$, then $u^{r^{s_i}}\neq u^{r^{s_j}}$ for all $1\leq i<j\leq t\leq p$.
\end{lemma}

\begin{proof} Without loss of generality, we may assume that $u=g_1\ldots g_t$. Since $\varphi(T)$ is a minimal product-one sequence over $G/N\cong C_p$, we conclude that $t\leq p$ and $u=y^k\in \pi(T)$ for some $k\in [0,m-1]$. Let $g_1\ldots g_i=x^{s_i}y^{c_i}$, where $1\leq i\leq t$, $s_i\in[0,p-1]$ and $c_i\in [0, m-1]$. Since $u=x^{s_t}y^{c_t}=y^k$, we have $s_t=0$ and $c_t=k$. Let $u_i=(g_{1}\ldots g_{i})^{-1}u(g_1\ldots g_{i})$, for each $i\in [1,t]$. Then $u_t=u\in\pi(T)$ and $u_i=g_{i+1}g_{i+2}\ldots g_{t}g_1\ldots g_i\in \pi(T)$ for $i\in[1,t-1]$. Therefore,
$$u_i=(g_{1}\ldots g_{i})^{-1}u(g_1\ldots g_{i})=(x^{s_i}y^{c_i})^{-1} y^{k} (x^{s_i}y^{c_i})=x^{-s_i} y^{k} x^{s_i}=u^{r^{s_i}}$$ where $i\in [1,t]$. Thus $\pi(T)\supseteq \{u^{r^{s_{1}}}, u^{r^{s_{2}}}, \ldots, u^{r^{s_{t}}}\}$. We next prove $s_i\neq s_j$ for all $1\leq i<j\leq t$. Assume to the contrary that $s_i=s_j$ for some $1\leq i<j\leq t$. Then $g_{i+1}\ldots g_j=(g_1\ldots g_i)^{-1}(g_1\ldots g_j)=(x^{s_i}y^{c_i})^{-1}x^{s_j}y^{c_j}=x^{s_j-s_i}y^{c_j-c_ir^{s_j-s_i}}=y^{c_j-c_i}\in N$. Thus $\varphi(g_{i+1}\ldots g_j)=1$, yielding a contradiction to the proposition that $\varphi(T)$ is a minimal product-one sequence over $G/N$. This proves the first result of the lemma. The final statement follows immediately from Lemma~\ref{basic}~(iii).
\end{proof}
We remark that the following lemma is a slight generalization of \cite[Lemma 16]{Bass2007} and the same proof for \cite[Lemma 16]{Bass2007} carries over. Let $N_i=x^iN$ be the $i$th coset of $N$ in $G$ for $0\leq i\leq p-1$.

\begin{lemma}\label{conjugationtwo}
Let $T_0$ be a sequence of $p$ elements in $N_i$ for some $i\in [1,p-1]$. For every $j\in [1,\ell]$, let $T_j$ be a sequence over $G$ such that $\pi(T_j)\cap N\neq \emptyset$, and let $u_j\in \pi(T_j)\cap N$. Then, for every $t\in[1,\ell]$, $\pi(T_0\bm\cdot T_1\bm\cdot\ldots\bm\cdot T_{t})$ contains the product set $\pi(T_0)\{u_1,u_1^{r},\ldots,u_1^{r^{p-1}}\}\ldots \{u_t,u_t^{r},\ldots,u_t^{r^{p-1}}\}$.
\end{lemma}

\begin{remark}\label{mequiv1}
$m\equiv 1\pmod p$.
\end{remark}
\begin{proof}
Since $r^{p}\equiv1 \pmod m$, we have $r^{p}\equiv1 \pmod q$ for every prime $q|m$. Since $\mbox{gcd}(r-1,m)=1$, we have $r\not\equiv1 \pmod q$. Therefore, $r$ has order $p$ modulo $q$ and thus $p|q-1$ for every prime $q|m$. Now, clearly $m\equiv1 \pmod p$.
\end{proof}

The next result gives the exact value of the small Davenport constant $\mathsf d(G)$.

\begin{proposition}\label{d(G)} $\mathsf d(G)=m+p-2$.
\end{proposition}

\begin{proof} Let $S_0=x^{[p-1]}\bm\cdot y^{[m-1]}$ be a sequence over $G$ with length $m+p-2$. Then $1\notin\Pi(S_0)$, so we have $\mathsf d(G)\geq m+p-2$. Next, let $S$ be any sequence over $G$ of length $m+p-1$. We show that $1\in\Pi(S)$. Assume to the contrary that $1\notin\Pi(S)$. Then $1\notin \Pi(S_N)$. This together with Lemma \ref{zerofree} implies that $m-1\geq |\Pi(S_N)|\geq |S_N|$. Let $T=S\bm\cdot S_N^{[-1]}$. Then $|T|=|S|-|S_N|\geq p$. Since $|\varphi(T)|=|T|\geq p$ and $\mathsf d(G/N)=\mathsf d(C_p)=p-1$, there exists a subsequence $T_1\mid T$ such that $\varphi(T_1)$ is a minimal product-one sequence over $G/N$. Let $t$ be the maximal integer such that $T_1\bm\cdot \ldots \bm\cdot T_t\mid T$ and $\varphi(T_i)$ is a minimal product-one subsequence over $G/N$. Since $\varphi(T\bm\cdot(T_1\bm\cdot \ldots \bm\cdot T_t)^{[-1]})$ is product-one free and $\mathsf d(G/N)=p-1$, we have $|T\bm\cdot(T_1\bm\cdot \ldots \bm\cdot T_t)^{[-1]}|=|\varphi(T\bm\cdot(T_1\bm\cdot \ldots \bm\cdot T_t)^{[-1]})|\leq p-1$.

Let $u_i\in \pi(T_i)$ for all $i\in [1,t]$. Since $1\notin \Pi(S)$, we have $1\notin\pi(T_i)$ and thus $$u_i\neq 1 \mbox{ for all } i\in[1,t].\ \ \ \ \ \ \ \ \ \ \ \ \ \ (\Delta)$$  By Lemma~\ref{conjugationone}, $$\pi(T_i)\supseteq \{u_i^{r^{s_{i1}}},u_i^{r^{s_{i2}}},\ldots,u_i^{r^{s_{it_i}}}\}$$ for all $i\in [1,t]$, $t_i=|T_i|$, and moreover, $u_i^{s_{ij}}\neq u_i^{s_{ik}}$ for $1\leq j<k\leq t_i$. Let $S_N=u_{t+1}\bm\cdot\ldots\bm\cdot u_{\ell}$ where $u_i\in N$, $i\in[t+1, \ell]$, $\ell=t+|S_N|$. Since $1\notin\Pi(S_N)$, $u_i\neq 1$ for all $i\in [t+1,\ell]$. This, together with $(\Delta)$, gives that $u_i\neq 1$ for all $i\in [1,\ell]$. Let $A_1=\{ u_1^{r^{s_{11}}},u_1^{r^{s_{12}}},\ldots,u_1^{r^{s_{1t_1}}}\}$, $A_i=\{1, u_i^{r^{s_{i1}}},u_i^{r^{s_{i2}}},\ldots,u_i^{r^{s_{it_i}}}\}$ where $i\in [2,t]$, and $A_i=\{1, u_i\}$ where $i\in [t+1,\ell]$. By Lemma~\ref{conjugationone}, we have
\[
 |A_i|= \left\{ \begin{array}{lll} & |T_1|, &i=1; \\ &  |T_i|+1, &i\in[2, t]; \\ & 2, &i\in[t+1, \ell].
\end{array} \right.
\]
Thus
\begin{align*}
\sum_{i=1}^{\ell}|A_i|&=\sum_{i=1}^{t}|T_i|+t-1+2(\ell-t)\\
&\geq (|T|-p+1)+\ell+|S_N|-1 \ \   (\mbox{as } |S_N|=\ell-t)\\
&=|S|+\ell-p\ \   (\mbox{as } |S|=|T|+|S_N|)\\
&= m+\ell-1 \ \   (\mbox{as } |S|=m+p-1).
\end{align*}

Let $\mathbf{A}=(A_1,\ldots, A_{\ell})$ and $M=\mbox{stab}(\Pi^{\ell}(\mathbf{A}))$. Since $A_i\subseteq N$ for all $i\in [1,\ell]$, we have $\Pi^{\ell}(\mathbf{A})\subseteq N$ and $M\subseteq N$. By Lemma~\ref{genKneser},
\begin{align*}
|\Pi^{\ell}(\mathbf{A})|&\geq |M|\Big(1-\ell+\sum_{Q\in N/M}\min\{\ell,|\{i\in[1,\ell]: A_i\cap Q\neq \emptyset\}|\}\Big) \\
&= |M|\Big(1-\ell+\sum_{Q\in N/M} |\{i\in[1,\ell]: A_i\cap Q\neq \emptyset\}|\Big) \\
&= |M|\Big(1-\ell+\sum_{Q\in N/M}\sum_{i\in[1,\ell], A_i\cap Q\neq \emptyset}1\Big)\\
&= |M|\Big(1-\ell+\sum_{i\in[1,\ell]}\sum_{Q\in N/M, A_i\cap Q\neq \emptyset}1\Big).
\end{align*}
Let $I_M$ be the subset of $[1,\ell]$ such that $i\in I_M$ if and only if $A_i\subseteq M$. Since $M$ is a cyclic subgroup of $N$, if $|I_M|\geq |M|$, then there exists a subset $\{i_1,\ldots, i_v\}\subseteq I_M$ such that $1=u_{i_1} \ldots u_{i_v}\in \pi(T_{i_1}\bm\cdot \ldots \bm\cdot T_{i_v})$, yielding a contradiction. So we may always assume that $|I_M|\leq |M|-1$. We now show $1\in\Pi^{\ell}(\mathbf{A})$. Clearly, if $M=N$, then $\Pi^{\ell}(\mathbf{A})=N$ and thus $1\in\Pi^{\ell}(\mathbf{A})$ as desired. We may always assume that $M\lneq N$. By Lemma~\ref{basic}~(i) and (ii), if any two elements of $A_i$ are contained in the same coset of $M$, then $A_i\subseteq M$ where $i\in [1,\ell]$; i.e., $i\in I_M$. This means if $i\notin I_M$, then all the elements of $A_i$ are in $|A_i|$ different cosets of $M$. Therefore,
\begin{align*}
|\Pi^{\ell}(\mathbf{A})|&\geq |M|\Big(1-\ell+\sum_{i\in[1,\ell]\setminus I_M}\sum_{Q\in N/M, A_i\cap Q\neq \emptyset}1+\sum_{i\in I_M}\sum_{Q\in N/M, A_i\cap Q\neq \emptyset}1\Big)\\
&= |M|(1-\ell+\sum_{i\notin I_M}|A_i|+|I_M|)\\
&= |M|(1-\ell+\sum_{i=1}^{\ell}|A_i|-\sum_{i\in I_M}|A_i|+|I_M|)\\
&\geq |M|(1-\ell+\sum_{i=1}^{\ell}|A_i|-p|I_M|)\\
&\geq |M|(m-p(|M|-1))\\
&=(|M|-1)(m-p|M|)+m\\
&\geq m \ \ \ \ (\mbox{as } M\lneq N, \mbox{ and thus } m> p|M|).
\end{align*}
Thus $1\in N=\Pi^{\ell}(\mathbf{A})$. Then there exists $u_i'\in A_i$ for each $i\in [1,\ell]$ such that $u_1'\ldots u_{\ell}'=1$. Let $I=\{i\in [1,\ell]: u_i'\neq 1\}=\{i_1,\ldots, i_k\}$, where $1\leq i_1<\ldots<i_k\leq \ell$ and $k=|I|$. Then $i_1=1\in I$ and thus $1=u_{i_1}'\ldots u_{i_k}'\in \pi(T_{i_1}\bm\cdot\ldots\bm\cdot T_{i_k})\subseteq\Pi(S)$, yielding a contradiction. This completes the proof of the lemma.
\end{proof}

We are now in a position to prove our main result.
\\

\noindent {\bf Proof of Theorem \ref{maintheorem}.}

It follows from Proposition~\ref{d(G)} that $\mathsf d(G)=m+p-2$. By \cite[Lemma 4]{ZG2005}, $\mathsf E(G)\geq |G|+\mathsf d(G)=pm+m+p-2$. Let $S$ be a sequence over $G$ with length $pm+m+p-2$. To prove our result, it is sufficient to show that $1\in \Pi_{pm}(S)$. We divide the proof into the following two cases.
\\

\noindent {\bf Case 1} $|S_{N_i}|\geq p$ for some $i\in [1,p-1]$.

Then there exists a subsequence $T_0\mid S_{N_i}$ with $|T_0|=p$. Clearly, $\varphi(T_0)$ is a minimal product-one subsequence over $G/N$. Let $\ell$ be the maximal integer such that
$$T_0\bm\cdot T_1\bm\cdot\ldots\bm\cdot T_{\ell}\mid S,$$
where $\varphi(T_j)$ is a product-one subsequence over $G/N$ with $|T_j|=p$ for every $j\in[1,\ell]$. By rearranging the order of $T_1, T_2, \ldots, T_{\ell}$ if necessary, we may assume that
$\pi(T_j)\neq \{1\}$ for all $j\in[1,v]$ and $\pi(T_j)=\{1\}$ for $j\in [v+1,\ell]$. Since $G/N\cong C_p$, by the maximality of $\ell$ and Theorem~\ref{EGZ}, we have $|S\bm\cdot (T_0\bm\cdot T_1\bm\cdot\ldots\bm\cdot T_{\ell})^{[-1]}|=|\varphi(S)|-|\varphi(T_0\bm\cdot T_1\bm\cdot\ldots\bm\cdot T_{\ell})|\leq \mathsf E(G/N)-1=2p-2$. Therefore, $p\ell=|\varphi(T_0\bm\cdot T_1\bm\cdot\ldots\bm\cdot T_{\ell})|-p\geq |S|-2p-2-p\geq mp+m-2p$. By Remark~\ref{mequiv1}, $m\equiv 1\pmod p$. Thus $\ell\geq m+(m-1)/p-1$. Let $1\neq u_j\in\pi(T_j)$ for all $j\in [1,v]$ and $1=u_j\in\pi(T_j)$ for $j\in [v+1,\ell]$. By Lemma~\ref{conjugationtwo}, $$\pi(T_0\bm\cdot T_1\bm\cdot\ldots\bm\cdot T_{\ell})\supseteq \pi(T_0)\{u_1,u_1^{r},\ldots,u_1^{r^{p-1}}\}\ldots\{u_{\ell},u_{\ell}^{r},\ldots,u_{\ell}^{r^{p-1}}\}.$$ Let $A_0=\pi(T_0)$, $A_j=\{u_{j},u_{j}^{r},\ldots,u_{j}^{r^{p-1}}\}$ for $j\in [1,v]$, and $A_j=\{1\}$ for $j\in [v+1,\ell]$. Let $\mathbf{A}=(A_1,\ldots, A_{\ell})$. Then $$\Pi_{pm}(S)\supseteq A_0\Pi^{m-1}(\mathbf{A}).$$ Let $M=\mbox{stab}(\Pi^{m-1}(\mathbf{A}))$. By Lemma~\ref{genKneser}, $$|\Pi^{m-1}(\mathbf{A})|\geq |M|\Big(2-m+\sum_{Q\in N/M}\min\big\{m-1,|\{j\in[1,\ell]: A_j\cap Q\neq \emptyset\}|\big\}\Big).$$
Let $I_M$ be the subset of $[1,\ell]$ such that $j\in I_M$ if and only if $A_j\subseteq M$. Since $M$ is a cyclic subgroup of $N$, if $|I_M|\geq (m/|M|) |M|+|M|-1=m+|M|-1$, then by using Theorem~\ref{EGZ} repeatedly for $m/|M|$ times, we can find a subset $\{j_1,\ldots, j_m\}\subseteq I_M$ such that $1\in \pi(T_{j_1}\bm\cdot \ldots\bm\cdot T_{j_m})$. Since $|T_{j_1}\bm\cdot \ldots\bm\cdot T_{j_m}|=pm$, $1\in\Pi_{pm}(S)$ and we are done. We only need to consider the case when $|I_M|\leq m+|M|-2$.

We now show $|\Pi^{m-1}(\mathbf{A})|\geq m$. Note that if $M=N$, then $\Pi^{m-1}(\mathbf{A})=N$ and thus $|\Pi^{m-1}(\mathbf{A})|\geq m$ as desired. We may always assume that $M\lneq N$. Let $Q\in N/M$ and $V_Q=\{j\in[1,\ell]: A_j\cap Q\neq \emptyset\}$. Clearly, $V_M\supseteq I_M$, and by Lemma~\ref{basic}~(i) we have $V_M\subseteq I_M$, whence $V_M=I_M$. Moreover, if $Q\neq M$, then $V_Q\cap V_M=V_Q\cap I_M=\emptyset$. By Lemma~\ref{basic}~(ii), if $j\in V_Q$ for $Q\neq M$, then all the elements of $A_j$ are in $p$ different cosets of $M$.  Let
$$\mu=|\{Q\in N/M: |V_Q|\geq m\}|.$$
If $\mu=0$, then as in the proof of Proposition~\ref{d(G)}, we have
\begin{align*}
|\Pi^{m-1}(\mathbf{A})|\geq& |M|\Big(2-m+\sum_{i\in[1,\ell]\setminus I_M}\sum_{Q\in N/M, A_i\cap Q\neq \emptyset}1+\sum_{i\in I_M}\sum_{Q\in N/M, A_i\cap Q\neq \emptyset}1\Big)\\
\geq& |M|(2-m+p(\ell-|I_M|)+|I_M|)\\
=& 2-m+p\ell-(p-1)|I_M|\\
\geq& 2-m+(pm+m-p-1)-(p-1)(m-1))\\
&(\mbox{as } \ell\geq m+(m-1)/p-1 \mbox{ and } |I_M|=|V_M|\leq m-1)\\
\geq & m.
\end{align*}

If $\mu\geq 2$, then
\begin{align*}
|\Pi^{m-1}(\mathbf{A})|\geq |M|(2-m+2(m-1))\geq m.
\end{align*}

If $\mu=1$, then let $R\in N/M$ be the unique coset of $M$ such that $|V_R|\geq m$. Assume that $R\neq M$. Then $R=\alpha M$ for some $\alpha\in N\setminus M$. Let $\alpha_j\in A_j \cap R=A_j \cap \alpha M$ for all $j\in V_R$. Since $\alpha_j\in A_j=\{u_j, u_j^{r}\ldots, u_j^{r^{p-1}}\}$, $\alpha_j^{r}\in A_j$. Thus $\alpha_j^{r}\in A_j\cap \alpha^{r}M$ for all $j\in V_R$. Implying that for all $j\in V_R$, $j\in V_{\alpha^{r}M}$ and thus $|V_{\alpha_j^rM}|\geq |V_R|$. Since $\alpha\notin M$, by Lemma~\ref{basic}~(ii), $\alpha M\neq \alpha^r M$. So we have found another coset $\alpha^r M(\neq R)$ such that $|V_{\alpha_j^rM}|\geq |V_R|\geq m$, yielding a contradiction to $\mu=1$. Thus we must have $R=M$. Since $|V_{\{1\}}|=|I_{\{1\}}|\leq m-1$ and $|V_M|=|V_R|\geq m>|V_{\{1\}}|$, we have $M\neq \{1\}$. By Lemma~\ref{genKneser},
\begin{align*}
|\Pi^{m-1}(\mathbf{A})|&\geq |M|\Big(2-m+\sum_{Q\in N/M}\min\{m-1,|\{j\in[1,\ell]: A_j\cap Q\neq \emptyset\}|\}\Big) \\
&= |M|\Big(2-m+\sum_{Q\in N/M, Q\neq M} |\{j\in[1,\ell]: A_j\cap Q\neq \emptyset\}|+m-1\Big) \\
&= |M|\Big(1+\sum_{Q\in N/M} |\{j\in[1,\ell]: A_j\cap Q\neq \emptyset\}|-|V_M|\Big) \\
&= |M|\Big(1+\sum_{Q\in N/M}\sum_{j\in[1,\ell], A_j\cap Q\neq \emptyset}1-|I_M|\Big)\\
&= |M|\Big(1+\sum_{j\in[1,\ell]}\sum_{Q\in N/M, A_j\cap Q\neq \emptyset}1-|I_M|\Big)\\
&= |M|\Big(1+\sum_{j\in[1,\ell]\setminus I_M}\sum_{Q\in N/M, A_j\cap Q\neq \emptyset}1+\sum_{j\in I_M}\sum_{Q\in N/M, A_j\cap Q\neq \emptyset}1-|I_M|\Big)\\
&= |M|(1+p(\ell-|I_M|))\\
&\geq |M|(pm+m-p-p(m+|M|-2))\\
&\ \ \ \ (\mbox{as } \ell\geq m+(m-1)/p-1 \mbox{ and } |I_M|\leq m+|M|-2)\\
&= |M|(m-p(|M|-1))\\
&= (|M|-1)(m-p|M|)+m \\
&\geq m \ \ \ \ (\mbox{as } M\lneq N, \mbox{ and thus } m> p|M|).
\end{align*}
So in all the cases, we have shown $|\Pi^{m-1}(\mathbf{A})|\geq m$. Thus $1\in N= \Pi^{m-1}(\mathbf{A})= A_0\Pi^{m-1}(\mathbf{A})\subseteq\Pi_{pm}(S)$. This completes the proof of Case 1.
\\

\noindent {\bf Case 2} $|S_{N_i}|< p$ for all $i\in [1,p-1]$.

Then $|S\bm\cdot S_N^{[-1]}|=|S_{N_1}\bm\cdot\ldots\bm\cdot S_{N_{p-1}}|\leq (p-1)^2$. Let $n=\mathsf v_1(S)$, i.e., the number of times $1$ occurs in $S$. If $n\geq p+m-2$, then let $S'=S\bm\cdot 1^{[-(m+p-2)]}$ be the subsequence of $S$ obtained by removing $m+p-2$ terms of $1$ from $S$. By Proposition~\ref{d(G)}, $\mathsf d(G)=m+p-2$. Since $|S'|=|S|-(m+p-2)=pm\geq \mathsf d(G)+1$, we can find a product-one subsequence $T_1$ of $S'$. Let $t$ be the maximal integer such that $T_1\bm\cdot\ldots\bm\cdot T_t\mid S'$ and $T_i$ is product-one for $i\in [1,t]$. Then $|S'\bm\cdot(T_1\bm\cdot\ldots\bm\cdot T_t)^{[-1]}|\leq m+p-2$. Hence $pm-(m+p-2)\leq d\leq pm$ where $d=|T_1\bm\cdot\ldots\bm\cdot T_t|$, and thus $0\leq pm-d\leq m+p-2$. Therefore, $1\in \pi(T_1\bm\cdot\ldots\bm\cdot T_t)=\pi(T_1\bm\cdot\ldots\bm\cdot T_t\bm\cdot 1^{[pm-d]})$. Since $|T_1\bm\cdot\ldots\bm\cdot T_t\bm\cdot 1^{[pm-d]}|=pm$, $1\in \Pi_{pm}(S)$. Next we may always assume that $n\leq p+m-3$. Again we divide the rest of the proof of Case 2 into the following two subcases.
\\

\noindent {\bf Subcase 2.1} $p\leq |S\bm\cdot S_N^{[-1]}|\leq (p-1)^2$.

Since $\mathsf d(G/N)=\mathsf d(C_p)=p-1$ and $|S\bm\cdot S_N^{[-1]}|\geq p$, we can find a factorization
$$S\bm\cdot S_N^{[-1]}=W_1\bm\cdot \ldots\bm\cdot W_k \bm\cdot W'$$
with $k\geq 1$, where $\varphi(W_i)$ is a minimal product-one subsequence over $G/N$ for $i\in [1,k]$ and $\varphi(W')$ is product-one free over $G/N$. Since $\mathsf d(G/N)=\mathsf d(C_p)=p-1$, we have $|W_k|\leq p$ and $|W'|\leq p-1$. Let $W_k=w_1\bm\cdot \ldots \bm\cdot w_e$ and $w_e=x^iy^j$, where $e=|W_k|$, $1\leq i\leq p-1$ and $0\leq j\leq m-1$. Let $S_N=u_1\bm\cdot\ldots\bm\cdot u_{\ell}$ such that $u_t\neq 1$ for $t\in[1,\ell-n]$ and $u_t=1$ for $t\in[\ell-n+1, \ell]$, where $\ell=|S_N|$. Then
$$pm+m+p-2-(p-1)^2\leq \ell=|S_N|=|S|-|S\bm\cdot S_N^{[-1]}|\leq pm+m-2.$$

We insert each $u_t$ into the product $w_1\ldots w_e$, either before or after $w_e$, where $t\in[1,\ell-n]$. If we put $u_t$ after $w_e$ then it multiplies the product by $u_t$; putting it before $w_e$ multiplies the product by $u_t^{r^i}$. Therefore, we have
$$\pi(W_k\bm\cdot u_1\bm\cdot\ldots\bm\cdot u_{\ell-n})\supseteq (w_1\ldots w_e)\{u_1, u_1^{r^i}\}\ldots\{u_{\ell-n}, u_{\ell-n}^{r^i}\}.$$
Let $A_0=\{w_1\ldots w_e\}$, $A_t=\{u_t, u_t^{r^i}\}$ for $t\in[1,\ell-n]$ and $A_t=\{1\}$ for $t\in[\ell-n+1,\ell]$. Let $\mathbf{A}=(A_1,\ldots, A_{\ell})$, $s=pm-|W_1\bm\cdot \ldots\bm\cdot W_k|$ and $M=\mbox{stab}(\Pi^{s}(\mathbf{A}))$. Since $pm-(p-1)^2\leq pm-|W_1\bm\cdot \ldots\bm\cdot W_k|= pm-(|S|-\ell-|W'|)\leq \ell-m+1$, we have
$$m+p-3\leq pm-(p-1)^2\leq s\leq \ell-m+1\leq pm-1.$$
By Lemma~\ref{genKneser},
$$|\Pi^{s}(\mathbf{A})|\geq |M|\Big(1-s+\sum_{Q\in N/M}\min\big\{s,|\{j\in[1,\ell]: A_j\cap Q\neq \emptyset\}|\big\}\Big).$$
Recall that $I_M$ is the subset of $[1,\ell]$ such that $j\in I_M$ if and only if $A_j\subseteq M$. As before, if $|I_M|\geq pm+|M|-1$, then there exists a product-one subsequence of $S_M$ with length $pm$ and we are done. We only need to consider the case when $|I_M|\leq \min\{\ell, pm+|M|-2\}$.

We first prove $|\Pi^{s}(\mathbf{A})|\geq m$. As before, we may always assume that $M\lneq N$. Recall that $V_Q=\{j\in[1,\ell]: A_j\cap Q\neq \emptyset\}$ where $Q\in N/M$. Clearly, $V_M\supseteq I_M$, and by Lemma~\ref{basic}~(i) we have $V_M\subseteq I_M$, so $V_M=I_M$. Moreover, if $Q\neq M$, then $V_Q\cap V_M=V_Q\cap I_M=\emptyset$. By Lemma~\ref{basic}~(ii), if $j\in V_Q$ for $Q\neq M$, then the two elements of $A_j$ are in $2$ different cosets of $M$. Let $$\mu=|\{Q\in N/M: |V_Q|\geq s+1\}|.$$ If $\mu=0$, then as in Case 1, we have
\begin{align*}
|\Pi^{s}(\mathbf{A})|&\geq |M|\Big(1-s+\sum_{i\in[1,\ell]\setminus I_M}\sum_{Q\in N/M, A_i\cap Q\neq \emptyset}1+\sum_{i\in I_M}\sum_{Q\in N/M, A_i\cap Q\neq \emptyset}1\Big)\\
&\geq |M|(1-s+2(\ell-|I_M|)+|I_M|)\\
&\geq |M|(1+2(\ell-s)) \ \ \ \ \ (\mbox{as } |I_M|=|V_M|\leq s)\\
&\geq |M|(1+2(m-1))\ \ \ \ (\mbox{as } s\leq \ell-m+1)\\
&\geq m.
\end{align*}

If $\mu\geq 2$, then
\begin{align*}
|\Pi^{s}(\mathbf{A})|\geq |M|(1-s+2s)\geq m \ \ (\mbox{as } s\geq m).
\end{align*}

If $\mu=1$, then let $R\in N/M$ be the unique coset of $M$ such that $|V_R|\geq s+1$. If $R=M$, then $V_R=I_M$. Note that since $|I_{\{1\}}|\leq m+p-3\leq s$, we have $M=R\neq \{1\}$. Since $M\lneq N$, we have $|I_M|\leq \min\{\ell, pm+|M|-2\}=pm+|M|-2$. As in Case 1, we have
\begin{align*}
|\Pi^{s}(\mathbf{A})|&\geq |M|\Big(1+\sum_{j\in[1,\ell]\setminus I_M}\sum_{Q\in N/M, A_j\cap Q\neq \emptyset}1+\sum_{j\in I_M}\sum_{Q\in N/M, A_j\cap Q\neq \emptyset}1-|I_M|\Big)\\
&=|M|(1+2(\ell-|I_M|)) \ \ \ \ \ \ \\
&\geq |M|(1+2((pm+m+p-2-(p-1)^2)-(pm+|M|-2))\\
& \ \ \ (\mbox{as } \ell\geq pm+m+p-2-(p-1)^2  \mbox{ and } |I_M|\leq pm+|M|-2)\\
&= |M|(1+2(m+p-(p-1)^2-|M|))\\
&= (2|M|-1)(m-|M|+p-(p-1)^2)-(p-1)^2+p+m\\
&\geq (2p+1)(p(p+1)+p-(p-1)^2)-(p-1)^2+p+m \\
& \ \ \ \ \ (\mbox{as } |M|\geq p+1 \mbox{ and } m\geq (p+1)|M|)\\
&\geq m.
\end{align*}

If $R\neq M$, then $V_R\cap I_M=\emptyset$, so $|I_M|+|V_R|\leq \ell$. Therefore, by Lemma~\ref{genKneser},
\begin{align*}
|\Pi^{s}(\mathbf{A})|&\geq |M|\Big(1-s+\sum_{Q\in N/M}\min\{s,|\{j\in[1,\ell]: A_j\cap Q\neq \emptyset\}|\}\Big) \\
&= |M|\Big(1-s+\sum_{Q\in N/M, Q\neq R} |\{j\in[1,\ell]: A_j\cap Q\neq \emptyset\}|+s\Big) \\
&= |M|\Big(1+\sum_{Q\in N/M} |\{j\in[1,\ell]: A_j\cap Q\neq \emptyset\}|-|V_R|\Big) \\
&= |M|\Big(1+\sum_{Q\in N/M}\sum_{j\in[1,\ell], A_j\cap Q\neq \emptyset}1-|V_R|\Big)\\
&= |M|\Big(1+\sum_{j\in[1,\ell]}\sum_{Q\in N/M, A_j\cap Q\neq \emptyset}1-|V_R|\Big)\\
&= |M|\Big(1+\sum_{j\in[1,\ell]\setminus I_M}\sum_{Q\in N/M, A_j\cap Q\neq \emptyset}1+\sum_{j\in I_M}\sum_{Q\in N/M, A_j\cap Q\neq \emptyset}1-|V_R|\Big)\\
&= |M|(1+2(\ell-|I_M|)+|I_M|-|V_R|)\\
&\geq |M|(\ell+1) \ \ \ \ (\mbox{as } |I_M|+|V_R|\leq \ell)\\
&\geq m \ \ \ \ (\mbox{as } s\leq \ell-m+1).
\end{align*}

In all the cases, we have shown $|\Pi^{s}(\mathbf{A})|\geq m$, whence $A_0\Pi^{s}(\mathbf{A})=\Pi^{s}(\mathbf{A})=N$. Hence, we can find a subsequence $U\mid u_1\bm\cdot\ldots\bm\cdot u_{\ell}$ with length $|U|=s$ such that $1\in \pi(W_1\bm\cdot \ldots\bm\cdot W_k \bm\cdot U)$. Since $|W_1\bm\cdot \ldots\bm\cdot W_k \bm\cdot U|=pm$, we have $1\in \Pi_{pm}(S)$. This completes the proof of Subcase 2.1.
\\

\noindent {\bf Subcase 2.2} $|S\bm\cdot S_N^{[-1]}|\leq p-1$.

Note that $|S_N|=|S|-|S\bm\cdot S_N^{[-1]}|\geq pm+m+p-2-(p-1)=pm+m-1$. By using Theorem~\ref{EGZ} repeatedly for $p$ times, we can find $p$ disjoint product-one subsequences $T_1, \ldots, T_p$ of $S_N$ with length $|T_j|=m$ for all $j\in [1,p]$. Therefore, $T_1\bm\cdot\ldots\bm\cdot T_p$ is a product-one subsequence of $S$ with length $pm$. This completes the proof of Subcase 2.2.                                                 \qed
\\

\noindent {\bf Acknowledgements}. This work was carried out during a visit of the first author to Brock University
as an international visiting scholar. He would like to sincerely thank the host institution for its hospitality
and for providing an excellent atmosphere for research. This work was supported in part by the National Science
Foundation of China (No. 11701256, 11871258), the Youth Backbone Teacher Foundation of Henan's University (No.
2019GGJS196), the China Scholarship Council (Grant No. 201908410132), and it was also supported in part by a
Discovery Grant from the Natural Sciences and Engineering Research Council of Canada (Grant No. RGPIN 2017-03903).

\end{document}